\documentclass[11pt,a4paper]{amsart}
\usepackage{amsmath,amsthm,amssymb,amsfonts}
\usepackage{graphicx}
\usepackage{epstopdf}
\usepackage{color}
\usepackage{hyperref}
\usepackage{paralist}

\usepackage[centering,margin=2.5cm]{geometry}
\usepackage{enumitem}
\usepackage{mathabx}%
\newtheorem{theorem}{Theorem}%

\theoremstyle{remark}

\theoremstyle{definition}

\def\aa{\alpha}
\def\bb{\beta}

\DeclareMathOperator{\pyr}{pyr}

\newcommand{\car}[2]{\Delta_{#1}\times\Delta_{#2}}

\newcommand{\psum}[3]{\pyr_{#1}(\Delta_{#2}\oplus\Delta_{#3})}

\newcommand{\D}[2]{D(#1,#2)}
\newcommand{\Dab}{\D{\alpha}{\beta}}
\newcommand{\K}[2]{K(#1,#2)}
\newcommand{\Kab}{\K{\alpha}{\beta}}
\newcommand{\F}[2]{\Phi(#1,#2)}
\newcommand{\f}[2]{\phi(#1,#2)}

\title{Polytopes with few vertices and few facets}

\author{Arnau Padrol}
\address{Institut f\"ur Mathematik, Freie Universit\"at Berlin}
\email{arnau.padrol@fu-berlin.de}
\thanks{This research is supported by the DFG
Collaborative Research Center SFB/TR~109 ``Discretization
in Geometry and Dynamics''.}

\begin{document}

\begin{abstract}
In this note we prove that the number of combinatorial types of $d$-polytopes with $d+1+\aa$ vertices and $d+1+\bb$ facets is bounded by a constant independent of $d$.
\end{abstract}

\maketitle

There is only one combinatorial type of $d$-polytope with $d+1$ vertices, the simplex. Every $d$-polytope with $d+2$ vertices is combinatorially equivalent to a repeated pyramid over a free sum of two simplices, $\psum{k}{n}{m}$ with $k\geq0$, $m,n\geq 1$ (cf. \cite[Section~6.1]{Gruenbaum}). The number of $d$-polytopes with $d+3$ vertices is exponential in $d$ \cite{Fusy2006}, and
 the number of $d$-polytopes with $d+4$ vertices is already superexponential in $d$ \cite{Padrol2013,Shemer1982}. Of course, the same numbers apply for polytopes with few facets, by polarity.

In this note we prove that, in contrast, there are few (combinatorial types of) polytopes that have both few vertices and few facets.

\begin{theorem}\label{thm:Kfvff}
For each pair of nonnegative integers $\aa$ and $\bb$ there is a constant $\Kab$, independent from $d$, such that the number of combinatorial types of $d$-polytopes with no more than $d+1+\aa$ vertices and no more than $d+1+\bb$ facets is bounded above by $\Kab$.
\end{theorem}

This theorem is a direct consequence of the following structural result.

\begin{theorem}\label{thm:Dfvff}
For each  pair of nonnegative integers $\aa$ and $\bb$ there is a constant $\Dab$ such that every $d$-polytope with no more than $d+1+\alpha$ vertices and no more than $d+1+\beta$ facets is a join of a simplex and an at most $\Dab$-dimensional polytope.
 
Equivalently, every $d$-polytope with $d> \Dab$ either is a pyramid, has more than $d+1+\alpha$ vertices or has more than $d+1+\beta$ facets.

Moreover, $\Dab$ satisfies

\[\f{\aa}{\bb}\leq\Dab\leq \min\left\{\F{\aa}{\bb},\F{\bb}{\aa}\right\},\]
where
\[\F{x}{y}=
\begin{cases}
0 & \text{ if }x=0\text{ or }y=0,\\
3x+y-2 & \text{ if }1\leq x\leq 5,\\
\binom{x}{2}+y+3& \text{ if }x\geq 5;
\end{cases}\]
and

\[\f{x}{y}=
\begin{cases}
0 & \text{ if }x=0\text{ or }y=0,\\
x+y& \text{ if }x=1\text{ or }y=1,\\
x+2y-1&\text{ if }x\geq y>1,\\
2x+y-1&\text{ if }y\geq x>1.
\end{cases}\]

\end{theorem}

Indeed, this theorem shows that for every $d$ the number of combinatorial types of $d$-polytopes with no more than $d+1+\aa$ vertices and no more than $d+1+\bb$ facets is bounded above by those in dimension $\Dab$. Since the vertex-facet incidences determine the combinatorial type, we get the following crude estimate for $\Kab$:
\[\Kab< 2 ^{(\Dab+\aa+1)(\Dab+\bb+1)}= 2 ^{O(\aa^4+\bb^4)}.
\]

Our proof of Theorem~\ref{thm:Dfvff} is based on a result of Marcus on minimal positively $2$-spanning configurations \cite{Marcus1981,Marcus1984}, which via Gale duality provides lower bounds on the number of vertices of what Wotzlaw and Ziegler call \emph{unneighborly polytopes}~\cite{WotzlawZiegler2011}.
A polytope $P$ is \emph{unneighborly} if for every vertex $v$ of $P$, there is some vertex $w$ such that $(v,w)$ does not form an edge of the graph of $P$.

\begin{theorem}[Marcus 1981{\cite{Marcus1981}}]\label{thm:marcus}
Let $P$ be an unneighborly $d$-polytope with $d + \aa + 1$ vertices
then
\[d\leq\begin{cases}
3\aa-1 & \text{ if }\aa\leq 5,\\
\binom{\aa}{2}+4& \text{ if }\aa\geq 5.
\end{cases}
\]
\end{theorem}

As Wotzlaw and Ziegler point out in \cite{WotzlawZiegler2011}, this upper bound is actually tight up to a constant factor. 
A slightly worse, but still quadratic, upper bound can also be deduced from \cite[Theorem~7.2.1]{WotzlawPhD}. This is a quantitative version of \emph{Perles' Skeleton Theorem}, a remarkable result first proved by Perles (unpublished, 1970)%
, reported by Kalai \cite{Kalai1994} and elaborated upon by Wotzlaw \cite[Part~II]{WotzlawPhD}.

\begin{proof}[Proof of Theorem~\ref{thm:Dfvff}]
To prove the upper bound on $\Dab$, we prove that if $P$ is a $d$-polytope with $d + \aa + 1$ vertices and $d+1+\bb$ facets and $d>\F{\aa}{\bb}$ then $P$ is a pyramid. The case $d>\F{\bb}{\aa}$ will then follow directly by polarity.

The proof is by induction on $\bb$. When $\bb=0$ (or $\aa=0$), then $P$ is a simplex and therefore a pyramid. If $\bb=1$, then $P$ is either a pyramid or a Cartesian product of two simplices, $P=\car{n-1}{m-1}$. 
In this case, $P$ is a polytope in dimension $n+m-2$ with $n\cdot m$ vertices and $n+m$ facets. For a fixed value of $\aa>0$, the maximal dimension is attained whenever either $n$ or $m$ is equal to~$2$. This proves that $d\leq \aa+1$.

Assume now that $\bb>1$, $\aa>0$ and  $d>\F{\aa}{\bb}$. By Theorem~\ref{thm:marcus}, $P$ is not unneighborly, and hence it has a vertex $v$ connected with an edge to all the other vertices. Therefore the vertex figure $P/v$ has $(d-1)+\aa+1$ vertices. Let $\bb'$ be such that $P/v$ has $(d-1)+\bb'+1$ facets.

If there is only one facet not containing $v$, then $P$ is a pyramid and we are done. Otherwise, 
$P$ has at least $(d-1)+\bb'+3$ facets because 
 every facet of $P/v$ contributes to a facet of $P$ containing~$v$ and there are at least two facets of $P$ not containing~$v$. 
Hence $\bb'\leq \bb-1$. Then, $d>\F{\aa}{\bb}$ implies that
\[\dim(P/v)=d-1> \begin{cases}
3\aa+\bb-2-1\geq 3\aa+\bb'-2& \text{ if }\aa\leq 5;\\
\binom{\aa}{2}+\bb+3-1\geq \binom{\aa}{2}+\bb'+3& \text{ if }\aa\geq 5.
\end{cases}\]
Therefore, by induction hypothesis, $P/v$ is a pyramid. Hence, all but one vertices of $P/v$ lie in a common hyperplane. This induces a hyperplane containing $v$ and all but one of its neighbors (in the graph of $P$). Call this vertex~$w$. Now, since $v$ is connected with an edge to all the remaining vertices of $P$, this means that all the vertices of $P$ but $w$ lie on a common hyperplane. Therefore, $P$ is a pyramid.

It only remains to prove the lower bounds, which are respectively attained by
\[\begin{cases}
\Delta_{\aa}\times\Delta_1& \text{ if }\aa\geq\bb=1,\\
\Delta_{\bb}\oplus\Delta_1& \text{ if }\bb\geq\aa=1,\\
\Asterisk_{\bb-1} (\Delta_1\oplus\Delta_1) \ast (\Delta_{\aa-\bb+1} \times \Delta_1) &\text{ if }\aa\geq \bb>1,\\
\Asterisk_{\aa-1} (\Delta_1\times\Delta_1) \ast (\Delta_{\bb-\aa+1} \oplus \Delta_1)  &\text{ if }\bb\geq \aa>1.
\end{cases}\]
where $\oplus$, $\ast$ and $\times$ represent free sum, join and Cartesian product, respectively; $\Asterisk_{k}$ denotes the iterated join of $k$ copies; and $\Delta_d$ is a $d$-simplex (see \cite{HenkRGZiegler} for the corresponding definitions). 
Indeed, $\Delta_{\aa}\times\Delta_1$ is $(\aa+1)$-dimensional with $2\aa+2$ vertices and $\aa+3$ facets, and is not a pyramid when $\aa\geq1$. Similarly, $\Asterisk_{\bb-1} (\Delta_1\times\Delta_1) \ast (\Delta_{\aa-\bb+1} \times \Delta_1)$ is a $(2\bb+\aa -1)$-dimensional polytope with $2\bb+2\aa$ vertices and $3\bb+\alpha$ facets that is not a pyramid. The cases $\Delta_{\bb}\oplus\Delta_1$ and $\Asterisk_{\aa-1} (\Delta_1\oplus\Delta_1) \ast (\Delta_{\bb-\aa+1} \oplus \Delta_1)$ follow by polarity.
\end{proof}

The lack of symmetry suggests that this upper bound is not optimal. Indeed, a small computation already shows that $\D{2}{\bb}=\bb+3$ for $2\leq \bb\leq 8$, which coincides with our lower bound for these cases. This raises the question: is $\Dab$ linear on both $\aa$ and $\bb$? This proof method cannot directly provide such a bound because, despite Marcus' original conjecture, the maximal dimension of an unneighborly polytope is quadratic in~$\aa$ (see \cite{WotzlawZiegler2011}).

\section*{Acknowledgements}
The author wants to thank G\"unter Ziegler for his useful suggestions and comments based on a preliminary version of this manuscript.

\bibliographystyle{amsplain}
\bibliography{fewVF}

\end{document}